\documentclass[12pt,leqno]{amsart}
\usepackage{latexsym,amsmath,amsthm,amssymb}
\usepackage[utf8]{inputenc}

\newtheorem{theorem}{Theorem}[section]
\newtheorem*{theorem*}{Theorem}
\newtheorem{lemma}[theorem]{Lemma}
\newtheorem{corollary}[theorem]{Corollary}%[section]
\newtheorem{proposition}[theorem]{Proposition}%[section]
\theoremstyle{remark}
\newtheorem{example}[theorem]{Example}
\newtheorem{remark}[theorem]{Remark}%[section]

\theoremstyle{definition}
\newtheorem{definition}[theorem]{Definition}

\newtheorem{problem}[theorem]{Problem}

\newcommand{\dN}{\ensuremath{\mathbb{N}}}

\newcommand{\dR}{\ensuremath{\mathbb{R}}}

\newcommand{\R}{\dR}
\newcommand{\N}{\dN}

\newcommand{\PP}{{\mathcal P}}

\begin{document}

\title{On the Monge-Kantorovich problem with additional linear constraints}
\author{Danila~Zaev}
\address{Faculty of Mathematics, Higher School of Economics, Moscow,  Russia}
\email{zaev.da@gmail.com}
\thanks{
The author is partially supported by AG Laboratory HSE, RF government grant, ag.  11.G34.31.0023.
}
%\date{\today}

\begin{abstract}
We consider the modified Monge-Kantorovich problem with additional restriction:
admissible transport plans must vanish on some fixed functional subspace. Different choice
of the subspace leads to different additional properties optimal plans need to satisfy.
Our main results are quite general and include several important examples. In particular, they include Monge-Kantorovich problems
in the classes of invariant measures and martingales. We formulate and prove a criterion for existence of a solution, a duality statement of the Kantorovich type, and
a necessary geometric condition on a support of optimal measure, which is analogues to the usual $c$-monotonicity.
\end{abstract}

\maketitle
%\tableofcontents

\section*{Introduction}

We are given probability measures $\mu_k$ on Polish spaces $X_k$, and choose a subspace $W$ of an appropriate functional space on $X_1\times\dots\times X_n$.
We consider the following optimization problem:
$$
\inf \left\{\int_{X_1\times\dots\times X_n} {c d\pi}: (\operatorname{Pr}_k)_\#\pi=\mu_k, \int \omega d\pi=0 ~ \forall \omega \in W \right\}
$$
for some cost function $c:X_1\times\dots\times X_n \rightarrow \R$.
This problem can be called the Monge-Kantorovich problem with additional linear constraints.
According to the usual terminology used in the Monge-Kantorovich theory, measures $\pi$ on $X_1\times\dots\times X_n$ with given marginals
$(\operatorname{Pr}_k)_\#\pi=\mu_k$ are called transport plans.
Here we restrict the set of transport plans adding the following requirement:
$$
\int \omega d\pi=0 ~ \forall \omega \in W
$$
Such constraint is obviously linear.

In the \textbf{first section} we give the precise formulation of the problem
and define a special class of functions $C_L$ containing $W$ in all of our examples.
Applying general machinery of the measure theory we prove a simple criterion of existence of a solution.
Namely, under an appropriate regularity assumption for the cost function
an optimal plan exists if and only if the set of admissible measures is not empty.

The motivation for this study arose from the applications in statistical physics and finance which lead to modified Monge-Kantorovich
problems, where sets of admissible transport plans are restricted in some way. The examples are the restriction of being invariant with respect to an
action of some group or to have a martingale property. Both examples can be seen as the particular cases of the general problem defined above.

We develop the general approach and establish some results, which can be seen as analogues of the appropriate statements 
for the classical Monge-Kantorovich
problem (for the theory of the classical problem see \cite{Bog-Kol}, \cite{Vill1}). 
One of such statements is the following Kantorovich-type duality result:
$$
\inf_{\pi\in \Pi_W}{\int c d\pi}=\sup_{\omega \in W}\left\{\sum_{k=1}^n \int_{X_k}{f_k(x_k) d\mu_k}\mbox{, } \sum_{k=1}^n f_k + \omega \leq c\right\}.
$$
The precise statement and a proof of this equality is presented in \textbf{section 2}.

The other result about the general problem is a geometric property of a support of optimal plans.
It is known that any solution of the standard  transportation problem  must be supported by a $c$-monotone set.
We formulate a similar property, which depends additionally on the space $W$ and call it $(c,W)$-monotonicity.
The necessity of such property is proven in \textbf{section 3} as a consequence of the duality statement.

One of the most interesting examples of linear restrictions are given by martingale measures which naturally appear
in financial applications. Any probabilistic model of a price of a time-dependent financial asset can be viewed as a multi-marginal transport plan.
It is known from the theory of financial markets that we need to restrict the set of possible distributions by the condition of ``fair game'':
at a particular time the expectation of the next value in the raw is equal to the present observed value.
This restriction is called ``martingale condition'', and it additionally restricts the set of available transport plans which can be used for
modeling the financial market. We refer to \cite{Hobson}, where the described price model is based on
a certain optimization on the set of martingale transport plans.
In \cite{Beigl3}, \cite{Beigl2} the theory of transportation problem with martingale restriction was developed and some deep results were obtained.
In \textbf{section 4} we show that the martingale restriction is actually a linear one and deduce some results about the martingale problem.

Another example of interesting linear restrictions is the invariance with respect
to a continuous action of some group. Such problems can naturally appear in the ergodic theory (see, for example, \cite{Lopes})
or geometry (\cite{Moameni}).
If the cost function is invariant, it is known that solutions of the classical Monge-Kantorovich plans are also invariant \cite{Moameni}.
In this case restriction of invariance does not add anything. In the other case it significantly modifies the problem.
In \textbf{section 5} we deduce from the general theory some new results about this problem and especially 
about the case of a compact group of invariance.

\section{Formulation of the problem}

Let $X_1,...,X_n$ be Polish spaces with Borel $\sigma$-algebra on them, $X=X_1\times\dots\times X_n$, $\mu_1,...,\mu_n$ are fixed
probability measures on $X_1,...,X_n$ respectively, $\mu=(\mu_1,...,\mu_n)$ is a tuple of such measures.
We denote by $\PP(X)$ the set of Borel probability measures over $X$, by $\Pi(\mu)$ the set of measures on $X$ with given marginals.
Both sets are equipped with the topology of weak convergence.

Let us introduce the functional spaces:
$$
C_L(\mu_i)=\left\{f\in L^1(X_i,\mu_i)\cap C(X_i)\right\}
$$
of continuous absolutely integrable functions with topology induced by $L^1(X_i,\mu_i)$ norm, and
$$
C_L(\mu)=\left\{h \in C(X): \exists f=\sum_{i=1}^n f_i(x_i)\in \bigoplus_{i=1}^n C_L(\mu_i)\mbox{ s.t. } |h|\leq f\right\}
$$
equipped with the seminorm:
$$
\| h \|_L:=\sup_{\pi \in \Pi}\int |h|d\pi
$$
It should be noted, that a very similar functional space was presented and studied in \cite{Vershik} by Vershik, Petrov, and Zatitskiy.

\begin{proposition}
 $\| h \|_L$ is a well-defined seminorm.
\end{proposition}
\begin{proof}
It is obviously finite
$$
\sup_{\pi \in \Pi}\int |h|d\pi \leq \sum \int f_i d\mu_i < \infty
$$
non-negative and absolutely scalable.
It is remain to check subadditivity: for any $h,g \in C_L(\mu)$
$$
\sup_{\pi \in \Pi}\int |h + g|d\pi \leq \sup_{\pi \in \Pi}\int |h|d\pi + \sup_{\pi \in \Pi}\int|g|d\pi=\| h \|_L+\| g \|_L
$$
\end{proof}

\begin{remark}
It can be seen straightforward that any function $f$ from $C_L(\mu)$ is absolutely integrable
with respect to any transport plan $\pi \in \Pi(\mu)$:
$$\int |f|d\pi \leq \sup_{\gamma \in \Pi}{\int |f|d\gamma}=\|f\|<\infty$$
\end{remark}

Denote by $F:=\bigoplus_{i=1}^n C_L(\mu_i) \subset C_L(\mu)$. Fix an arbitrary subspace $W \subset C_L(\mu)$ and a function $c \in C_L(\mu)$.
We are interested in the Monge-Kantorovich problem:
$$
\inf_{\pi \in \Pi_W}\left\{\int_{X} c(x_1,...,x_n)d\pi\right\}
$$
where the infimum is taken over the set of optimal plans with the property $\pi|_W=0$:
$$
\Pi_W=\left\{ \pi \in \PP(X): \forall w\in W \int w d\pi=0, \operatorname{Pr}_\#(\pi)=\mu \right\}
$$
$\operatorname{Pr}$ here is the natural projection from $X$ on the tuple of spaces $(X_1,...,X_n)$.

Suppose $W\ni w=f_i\circ \operatorname{Pr}_i$, where $f_i \in C_L(\mu_i)$. Then
$$
0 = \int_{X} (f_i\circ \operatorname{Pr}_i) d\pi = \int_{X_i} f_i d\mu_i
$$
Hence we have a necessary condition for measures $\mu \in \PP(X)$ to have a transport plan in $\Pi_W(\mu)$:
$$
\int_X f_i d\mu_i = 0 \mbox{ if } (f\circ \operatorname{Pr}_i) \in W
$$
In other words: $\mu_k|_{W\cap C_L(\mu_k)} = 0$ for any $k=1,...,n$.

Thus we can formulate the following central problem.
\begin{problem}
 \emph{(Kantorovich problem with linear constraints)}

 Given some fixed Polish spaces $X=X_1\times\dots\times X_n$, measures $\mu_i \in \PP(X_i)$, cost function $c\in C_L(\mu)$,
 and a linear subspace $W \subset C_L(\mu)$, where $F:=\bigoplus_{i=1}^n C_L(\mu_i) \subset C_L(\mu)$ find
 $$
 \inf_{\pi \in \Pi_W(\mu)}\left\{\int_{X} c(x)d\pi\right\}
 $$
\end{problem}

Now we prove the important fact about the space $C_L(\mu)$.
It turns out that the natural injection of $C_b(X)$ in $C_L(\mu)$ is dense:

\begin{lemma}
 $C_b(X)$ is dense in $C_L(\mu)$.
\end{lemma}
\begin{proof}
 At first one can show that $\| \cdot \|_{C_L} \leq \| \cdot \|_{C_b}$ and hence the natural injection
 $C_b(X) \hookrightarrow C_L(\mu): f \rightarrow f$ is continuous.
 For any $h \in C_b(X)$
 $$
 \| h \|_L=\sup_{\pi \in \Pi}\int |h|d\pi \leq \sup_{x \in X}|h(x)|\cdot \sup_{\pi \in \Pi}\int d\pi=
 \sup_{x \in X}|h(x)|=\| h \|_{C_b}
 $$

 For the next step of the proof fix any $h\in C_L$, and let $|h|\leq f\in F=\bigoplus_{i=1}^n C_L(\mu_i)$. 
 Let $k \in \N$, $h^k=\min\{k,h\}$ and $h^k_k=\max\{\min\{k,h\},-k\}\in C_b(X)$. Note that $|h^k|\leq f$.
 Our goal is to show that $\|h-h^k_k\|_L \rightarrow 0$ as $k \rightarrow \infty$.
 $$
 \|h-h^k_k\|_L\leq\|h-h^k\|_L+\|h^k-h^k_k\|_L
 $$
 The fact $\pm h-k \leq f-k = \sum_{i=1}^n\left(f_i-\frac{k}{n}\right)$ together with positivity of the operator $(\cdot)_+:=\max\{\cdot,0\}$ implies
 $$
 \|h-h^k\|_L=\sup_{\pi \in \Pi}\int((h-k)_+)d\pi\leq \sum_{i=1}^n\int \left(f_i-\frac{k}{n}\right)_+d\mu_i\rightarrow 0, 
 \mbox{ as } k \rightarrow \infty
 $$
 $$
 \|h^k-h^k_k\|_L=\sup_{\pi \in \Pi}\int((-h^k-k)_+)d\pi\leq \sum_{i=1}^n\int\left(f_i-\frac{k}{n}\right)_+d\mu_i\rightarrow 0, 
 \mbox{ as } k \rightarrow \infty
 $$
 Convergence here is due to the Lebesgue dominated convergence theorem: 
 for any $k\in \N$ $\left(f_i-\frac{k}{n}\right)_+\leq |f_i|\in L_1(X_i,\mu_i)$ and $\left(f_i-\frac{k}{n}\right)_+\rightarrow 0$ pointwise.
 \end{proof}

It is a well-known consequence of the Prokhorov theorem that the set $\Pi(\mu)$ is compact in the topology of weak convergence.
Obviously the set $\Pi_W=\{\pi: \pi|_W=0\}$ is closed in such topology. Thus the set $\Pi_W(\mu)$ is also compact.
To establish the existence result we need to prove continuity of the functional: $\pi \rightarrow \int h d\pi$.
Fortunately it follows directly from the previous lemma.
\begin{corollary}
\label{continuity of functional}
The functional $\pi \rightarrow \int h d\pi$ from $\Pi(\mu)$ to $\R$ is continuous for any $h\in C_L(\mu)$.
\end{corollary}
\begin{proof}
We need to check that for any sequence of transport plans $(\pi_k)$ such that
$\lim_k \int \rho d\pi_k=\int \rho d\pi$ for any $\rho \in C_b$ we have $\lim_k\int h d\pi_k=\int h d\pi$.
Since $C_b$ is dense in $C_L$ there is a sequence $\rho_n \rightarrow h$ in $\|\cdot \|_L$ topology
and $\rho_n \in C_b$ for any $n \in \N$.
Note, that $\rho_n \rightarrow h$ in $\|\cdot \|_L$ means it tends to the limit uniformly with respect to the set of all transport plans.
This fact together with existence of the limits $\lim_k \int \rho_N d\pi_k$ and $\lim_n \int \rho_n d\pi_K$ for each sufficiently large $N$ and $K$
allows us to change the order of double limit in the following argument
$$
\lim_k\int h d\pi_k = \lim_k \lim_n \int \rho_n d\pi_k = \lim_n \lim_k \int \rho_n d\pi_k = \lim_n \int \rho_n d\pi = \int h d\pi
$$
\end{proof}

Compactness and continuity together implies the following existence criterion
\begin{proposition}
The Kantorovich problem with additional linear constraint has a solution if and only if $\Pi_W=\{\pi: \pi|_W=0\}$ is not empty.
\end{proposition}

\begin{remark}
One can replace the constraint $\pi|_W=0$ by $\pi|_{\bar{W}}=0$, where $\bar{W}$ is the closure of $W$ in the $\|\cdot\|_L$ topology.
By definition of such topology it is obvious that $\pi|_W=0 \iff \pi|_{\bar{W}}=0$, thus the replacing doesn't change anything.
\end{remark}

\section{Kantorovich duality}

The theorem below generalizes the well-known Kantorovich duality for the case of additional linear constraints.

\begin{theorem}
\label{Kantorovich duality}
Let $X_1,...,X_n$, $X=X_1\times \dots \times X_n$ be Polish spaces,
$\mu=\left(\mu_k \in \PP(X_k)\right)$ for $k=1,...,n$,
$W$ be a subspace of $C_L(\mu)$, $c \in C_L(\mu)$. Then
$$
\inf_{\pi\in \Pi_W}{\int c d\pi}=\sup_{f + \omega \leq c}{\sum_{k=1}^n \int_{X_k}{f_k(x_k) d\mu_k}}.
$$
where $f \in F=\bigoplus_{i=1}^n C_L(\mu_i)$, $f=\sum_{i=1}^n f_i(x_i)$, $\omega \in W$.
\end{theorem}

To prove this theorem we need some extra results. First, it is a version of Kantorovich duality theorem for the original problem:
\begin{theorem}
\label{Original Kantorovich duality}
Let $X_1,...,X_n$, $X=X_1\times \dots \times X_n$ be Polish spaces,
$\mu=\left(\mu_k \in \PP(X_k)\right)$ for $k=1,...,n$, $c \in C_L(\mu)$. Then
$$
\inf_{\pi\in \Pi}{\int c d\pi}=\sup_{f\leq c}{\sum_{k=1}^n \int_{X_k}{f_k(x_k) d\mu_k}}.
$$
where $f \in F=\bigoplus_{i=1}^n C_L(\mu_i)$.
\end{theorem}
Proof of the similar version of the theorem can be found in \cite{Rachev1} or \cite{kel1}. In the Appendix we provide its complete proof.

The next statement we are going to use is the general version of minimax theorem from (\cite{Adams}, Th. 2.4.1).
The proof and all explanations can also be found there.
\begin{theorem}
\label{minimax theorem}
Let $K$ be a compact convex subset of a Hausdorff topological vector space, $Y$ be a convex subset of an arbitrary vector space, and $h$ be a real-valued function ($\leq +\infty$) on $K\times Y$,
which is lower semicontinuous in $x$ for each fixed $y$, convex on $K$, and concave on $Y$. Then
$$
\min_{x \in K}\sup_{y \in Y} h(x,y) = \sup_{y \in Y} \min_{x \in K} h(x,y)
$$
\end{theorem}

Now we are ready to prove our result.
\begin{proof}
\emph{(of the Kantorovich duality theorem (\ref{Kantorovich duality}))}

The inequality
$$
\inf_{\pi\in \Pi_W}{\int c d\pi} \geq \sup_{f + \omega \leq c}{\sum_{k=1}^n \int {f_k d\mu_k}}
$$
is almost obvious:
\begin{multline*}
 \inf_{\pi\in \Pi_W}{\int c d\pi} \geq \inf_{\pi\in \Pi_W}\sup_{f+\omega \leq c}{\int (f+\omega) d\pi} =\\
 =\inf_{\pi\in \Pi_W}\sup_{f+\omega \leq c}{\sum_{i=1}^n\int f_i d\mu_i}=\sup_{f+\omega \leq c}{\sum_{i=1}^n\int f_i d\mu_i}
\end{multline*}

Prove the opposite inequality.
$$
\sup_{f + \omega \leq c}{\sum_{k=1}^n \int {f_k d\mu_k}}=\sup_{\omega \in W} \sup_{f \in F, f \leq (c-\omega)}{\sum_{k=1}^n \int {f_k d\mu_k}}=
\sup_{\omega \in W} \inf_{\pi\in \Pi}{\int (c-\omega) d\pi}
$$
Here we used theorem (\ref{Original Kantorovich duality}) for the cost function $(c-\omega) \in C_L(\mu)$.
The next step is to use minimax theorem (\ref{minimax theorem}) to interchange infimum and supremum.
Using the notation of that theorem assume $K=\Pi(\mu)$, $Y=W$, $h(\pi,\omega)=\int (c-\omega) d\pi$.
Note that $h(\pi,\omega)$ is linear in both arguments and continuous in $\pi$ for any fixed $\omega$
(it has been already proven by us, see (\ref{continuity of functional})). Hence all assumptions of the theorem are satisfied and we obtain
$$
\sup_{\omega \in W} \inf_{\pi\in \Pi}{\int (c-\omega) d\pi}=\inf_{\pi\in \Pi}\sup_{\omega \in W}{\int (c-\omega) d\pi}
$$
If $\pi \notin \Pi_W$, then there exists $\omega_1 \in W$ such that $\int \omega_1 d\pi < 0$.
The choice $\omega = \alpha \omega_1$, $\alpha \rightarrow +\infty$ shows that the supremum $\sup_{\omega \in W}{\int (c-\omega) d\pi}$ is $+\infty$.
Thus we conclude:
$$
\inf_{\pi\in \Pi}\sup_{\omega \in W}{\int (c-\omega) d\pi}=\inf_{\pi\in \Pi_W}{\int c d\pi}
$$
And it is exactly what we need to complete the proof.
\end{proof}

\begin{remark}
 Note that the proposition of the theorem remains true for the case of empty $\Pi_W(\mu)$ if we agree to define $\inf(\varnothing)=+\infty$.
\end{remark}

By the same argument the following version of the Kantorovich duality for continuous bounded functions can be obtained:
\begin{theorem}
\label{bounded Kantorivich duality}
Let $X_1,...,X_n$, $X=X_1\times \dots \times X_n$ be Polish spaces,
$\mu=\left(\mu_k \in \PP(X_k)\right)$ for $k=1,...,n$,
$W$ be a subspace of $C_b(\mu)$, $c \in C_b(\mu)$ Then
$$
\inf_{\pi\in \Pi_W}{\int c d\pi}=\sup_{f + \omega \leq c}{\sum_{k=1}^n \int_{X_k}{f_k(x_k) d\mu_k}}.
$$
where $f \in F=\bigoplus_{i=1}^n C_b(X_i)$, $f=\sum_{i=1}^n f_i(x_i)$, $\omega \in W$.
\end{theorem}
The only thing we need to change in the previous proof is to use instead of \ref{Original Kantorovich duality}
the version of classical Kantorovich duality with bounded continuous functions (for example, theorem (5.10) from (\cite{Vill1}).

\section{Geometry of optimal transport plans}
In this section we are going to formulate an analogue of c-monotonicity for our problem and prove its necessity for optimal transport plans.

\begin{definition}
 For two measures $\alpha$, $\beta$ on $X=X_1\times\dots\times X_n$ define the equivalence relation $\sim_W$:
 $\alpha \sim_W \beta \mbox{ iff}$
 \begin{enumerate}
  \item $(\operatorname{\operatorname{Pr}}_k)_{\#}(\alpha)=(\operatorname{\operatorname{Pr}}_k)_{\#}(\beta) \mbox{  } \forall k=1,...,n$
  \item $\int \omega d\alpha=\int \omega d\beta \mbox{  } \forall \omega \in W$
 \end{enumerate}

\end{definition}
We denote by $[\beta]_W$ the equivalence class of $\beta$ with respect to $\sim_W$.
Let $S_m$ be a set of $m$ points in the space $X$, $\beta_s$ be a measure with the support $S_m$.

\begin{definition}
For a Borel measurable cost function $c:X \rightarrow \R$ and a linear subspace $W \subset C_L(\mu)$
a set $\Gamma \subset X$ is called $(c,W)$-monotone if and only if for any $m \in \N$,
any $S_m \subset \Gamma$ any measure $\beta_s$, such that $\operatorname{supp}(\beta_s) = S_m$, and any measure $\alpha \sim_W \beta_s$:
$$
\int c d \beta_s \leq \int c d\alpha
$$
\end{definition}
\begin{remark}
Due to the linearity of integrals nothing is changed if we consider only probability measures $\beta_s$ in this definition.
\end{remark}
\begin{proposition}
 If $W=\{0\}$, then the notion of $(c,W)$-monotonicity is equivalent to the notion of usual $c$-monotonicity.
\end{proposition}
\begin{proof}
 $(c,\{0\})$-monotonicity obviously implies $c$-monotonicity.
 The converse statement follows from the well-known (see \cite{Vill1}) implication: c-monotonicity of $\operatorname{supp}(\beta)$ implies that $\beta$ is optimal in the class of
 measures on $X$ with the same marginals with $\beta$.
 Indeed, if $W=\{0\}$ then this class coincide with the equivalence class $[\beta]_W$. Thus $\beta$ is optimal in $[\beta]_W$ for any $\beta$ with
 the support consisted of the finite number of points and laid in the $c$-monotone set (here we also use the fact that a subset of c-monotone set is also c-monotone).
 \end{proof}

\begin{definition}
 A transport plan $\pi \in \Pi_W(X)$ is called $(c,W)$-monotone iff there is a $(c,W)$-monotone set $\Gamma$ of full $\pi$-measure:
 $\pi(\Gamma)=1$.
\end{definition}

\begin{theorem}
 Let $X_k$ $(k=1,...,n)$, $X=X_1\times\dots\times X_n$ be Polish spaces, $\mu=\left(\mu_k \in \PP(X_k)\right)$,
 $c \in C_L(\mu)$ is a cost function, $W \subset C_L(\mu)$ is a vector subspace, $\mu_k|_{W\cap C_L(\mu_k)}=0$, and $\pi_* \in \Pi_W(\mu)$
 is the minimizer of the primal Kantorovich problem with additional linear constraints:
 $$\inf_{\pi \in \Pi_W}\int_{X} c d\pi$$
 then $\pi_*$ is a $(c,W)$-monotone transport plan.
\end{theorem}
The following proof relies on the proven duality result from the previous section.
\begin{proof}
 By the Kantorovich duality statement (\ref{Kantorovich duality})
 $$\int_X c d\pi_* = \sup \left\{ \sum_{k=1}^n\int f_k d\mu_k\right\} $$
 where supremum runs among all pairs $(f, \omega) \in F\times W$ such that
 $f+\omega\leq c$. Recall the notation: $F=\bigoplus_{i=1}^n C_L(\mu_i)$, $f=\sum_{i=1}^n f_i(x_i)$
 Let $(f^{(k)},\omega_k)$ be a maximizing sequence in the dual problem and let $c_k=c-f^{(k)}-\omega_k$. Since
 $$
 \int_X c_k d\pi_*=\int_X c d\pi_*-\sum_{i=1}^n\int_{X_i} f^{(k)}_i d\mu_i \rightarrow 0
 $$
 and $c_k \geq 0$ we can find a subsequence $c_{k(j)}$ and a Borel set $\Gamma$ for which $\pi_*(\Gamma)=1$, such that $c_{k(j)} \rightarrow 0$ on $\Gamma$. In the following by slightly abuse of notation we will denote indices $\{k(j)\}$ simply by $\{j\}$.
 If $S=\{x_i\}_{i=1}^m \subset \Gamma$, $\beta_s$ is a measure with the support $S$ and $\alpha \in [\beta_s]_W$ we get
 $$
  \int c d\alpha \geq \sum_{i=1}^n\int f^{(k)}_i d\alpha + \int \omega_k d\alpha
 $$
 Since
 $$
  (\operatorname{Pr}_i)_{\#}(\alpha)=(\operatorname{Pr}_i)_{\#}(\beta) \mbox{ for all }i\in{1,...,m}
 \implies \int f^{(k)}_i d\alpha=\int f^{(k)}_i d\beta_s
 $$
 $$
  \int \omega d\alpha=\int \omega d\beta \mbox{ for all }\omega \in W \implies \int \omega_k d\alpha= \int \omega_k d\beta_s
 $$
 we obtain
 $$
 \int c d\alpha \geq \sum_{j=1}^n\int f^{(k)}_i d\beta_s +
  \int \omega_k d\beta_s  = \int (c - c_k) d\beta_s
 $$
 for any $k$. Letting $k \rightarrow \infty$ the $(c,W)$-monotonicity of $\Gamma$ follows.
\end{proof}
 Note, that the sufficiency of $(c,W)$-monotonicity is not established in the general case.

\section{Martingale Monge-Kantorovich problem}
 In this section we show that so-called martingale Monge-Kantorovich problem can be seen as a
 particular case of the Monge-Kantorovich problem with additional linear constraints.
 General theory, developed in previous chapters, implies the Kantorovich duality statement and some over known results about martingale optimal plans.
 %Let $Y$ be a Polish topological vector space (or, equivalently, a separable Freschet space)
 %and $X_1=X_2=...=X_n=Y$, $X=Y^n$. Consider $\mu_k \in \PP(Y)$ ($k \in \{1,...,n\})$ such that each of %them has a finite first moment. Define
 Let $X_1=X_2=...=X_n=\R$, $X=\R^n$. Consider $\mu_k \in \PP(\R)$ ($k \in \{1,...,n\})$ such that each of them has a finite first moment. Define
 $$
 W=\left\{\sum_{k=1}^{n-1}\rho(x_1,...,x_k)(x_{k+1}-x_k): \rho(x_1,...,x_k) \in C_b(X)\right\}
 $$
 \begin{proposition}
  In the defined setting $W \subset C_L(\mu)$
 \end{proposition}
 \begin{proof}
 $$
   \sum_{k=1}^{n-1}|\rho(x_1,...,x_k)(x_{k+1}-x_k)| \leq\sum_{k=1}^{n-1} C_{k,\rho} |x_{k+1}-x_k|
  \leq\sum_{k=1}^{n-1} C_{k,\rho} \left( |x_{k+1}| + |x_k| \right)
 $$
 $$
 \int C_{k,\rho} |x_k| d \mu_k <\infty \mbox{ for any } k\in\{1,...,n\}
 $$
 The last inequality follows from the fact that each of $\mu_k$ has a finite first moment.
 \end{proof}

 It can be easily proven that:
 \begin{proposition}
 $$
 \Pi_W(\mu_1,...,\mu_n)=\{\pi:\int_\R x_{k+1} d\pi_{x_1,...,x_k}(x_{k+1}) = x_k\}
 $$
 for almost all $x_k$ w.r.t. $(P_{1,...,k})\#(\pi)$, for any $k\in\{1,...,n\}$ where $\pi_{x_1,...,x_k}$ is a conditional measure.
 \end{proposition}
 \begin{proof}
  Indeed
  \begin{multline*}
  \int_\R x_{k+1} d\pi_{x_1,...,x_k}(x_{k+1}) = x_k \mbox{ for almost all } x_k \mbox{ w.r.t. } (P_{1,...,k})\#(\pi) \iff
  \\
  \iff \int \rho(x_1,...,x_k) \int x_{k+1} d\pi_{x_1,...,x_k}(x_{k+1}) d\mu_{k+1} = \\
  \int x_k \rho(x_1,...,x_k) d\mu_{k+1} \mbox{ for any } \rho \in C_b(X_1\times...\times X_k) \iff
  \\
  \iff \int \rho(x_1,...,x_k)(x_{k+1}-x_k) d\pi \mbox{ for any } \rho \in C_b(X_1\times...\times X_k)
  \end{multline*}
  Since the continuity is obvious, the proof is complete.
 \end{proof}

 This problem is called the \textit{martingale Kantorovich problem}.
 Note, that existence of the solution for this problem is not guaranteed.

 By substitution of particular form of $W$ in the general duality statement \ref{Kantorovich duality} we directly obtain the duality in the following form:
 \begin{theorem}
  In the martingale setting the following formulation of Kantorovich duality holds:
  \begin{multline*}
  \inf_{\pi\in \Pi_W}{\int_{R^n}}{c(x_1,...,x_n) d\pi}=\\
  =\sup \left\{\sum_{k=1}^n\int_{\R}{\phi(x_k) d\mu_k}: \sum_{k=1}^n\phi_k(x_k) + \sum_{k=1}^{n-1}\rho(x_1,...,x_k)(x_{k+1}-x_k) \leq c\right\}.
  \end{multline*}
  where $(\phi_k)_{k=1}^n \in (C_L(\mu_k))_{k=1}^n$, $\rho(x_1,...,x_k) \in C_b(\R^n)$.
 \end{theorem}
 This result was obtained by Beiglboeck, Henry-Labordere and Penkner in \cite{Beigl3}.
 Also they show (Prop. 4.1 in \cite{Beigl3}) that optimal value of the dual problem in this setting is not attained in general.

\section{Invariant Monge-Kantorovich problem}
 \label{Continuous invariant setting}

Let $G$ be some group acting continuously on $X_1,..., X_n$.
Suppose additionally that an action of $G$ is defined on $X$, and it is diagonal
$$g(x_1,...,x_n)=(g_1(x_1),...,g_n(x_n))$$
Define subspace $W$ as a subspace of $C_b$
$$W := \operatorname{span} \{h \circ g - h: g \in G, h \in C_b(\mu)\}$$
where 'span' is for the space of finite linear combinations.

One can obtain the following characterization of measures which are vanishing on $W$:
\begin{proposition}
 For any $\pi \in \PP(X)$ $\pi|_W=0$ if and only if $\pi$ is an invariant measure.
\end{proposition}
\begin{proof}
\begin{multline*}
 \pi \mbox{ is invariant } \iff \int{ h d g_\#\pi}=\int {h d \pi} \mbox{ }\forall h\in C_L(\mu), \forall g\in G \iff\\
 \iff \int{ h \circ g d \pi}=\int {h d \pi} \mbox{ }\forall h\in C_L(\mu), \forall g\in G \iff\\
 \iff \int{ h \circ g - h d \pi}=0 \mbox{ }\forall h\in C_L(\mu), \forall g\in G \iff \pi|_W=0
\end{multline*}
\end{proof}
If $W$ is defined in such way, we will refer to the problem:
 $$
 \inf_{\pi \in \Pi_W(\mu)}\left\{\int_{X} c(x)d\pi\right\}
 $$
as the \textit{invariant Kantorovich problem}.
Due to diagonality of the action of $G$ we have that $f\circ \operatorname{Pr}_k\circ g = f \circ g_k \circ \operatorname{Pr}_k$ and
$$
F\cap W=\bigoplus_{k=1}^n\{f\in C_L(\mu_k): f\circ g_k\circ P_k - f \circ P_k \in W \}
$$
It means that the necessary condition for invariant problem to have a solution is
the invariance of given marginal measures $\mu_k$ with respect to the action of $G$ on $X_k$.
It can be proved that such condition is also sufficient.
\begin{proposition}
Invariant Kantorovich problem with invariant marginals has a solution.
\end{proposition}
\begin{proof}
It is enough to show that $\mu \otimes \nu \in \Pi_W(\mu,\nu)$.
For any $h \in W$:
\begin{multline*}
\int(h\circ g - h) d(\otimes_{k=1}^n\mu_k) = \int_{X_2\times\dots\times X_n}\left(\int_{X_1} h(g(x))-h(x) d\mu_1(x_1)\right)d(\otimes_{k = 2}^n\mu_k)=\\
=\int_{X_2\times\dots\times X_n}\biggl(\int_{X_1} (h(g(x))-h(g_1(x_1),x_2,...x_n))+\\
+(h(g_1(x_1),x_2,...,x_n)-h(x)) d\mu_1(x_1)\biggr) d(\otimes_{k = 2}^n\mu_k)=\\
=\int_{X_2\times\dots\times X_n}\left(\int_{X_1} h(g(x))-h(g_1(x_1),x_2,...x_n) d\mu_1(x_1)\right)d(\otimes_{k = 2}^n\mu_k)=\\
=\int_{X_1}\left(\int_{X_2\times\dots\times X_n} h(g(x))-h(g_1(x_1),x_2,...x_n) d\otimes_{k = 2}^n\mu_k\right)d\mu_1(x_1)=...\\
...=\int_{X_1}\int_{X_2}\dots\int_{X_n} (h(g(x))-h(g_1(x_1),...,g_{n-1}(x_{n-1},x_n)) d\mu_n(x_n)...d\mu_1(x_1) = 0
\end{multline*}
\end{proof}
\begin{example}
Let $k=2$, $X_1=X_2=\R^{\N}$ be the direct product of countable number of $\R^1$, $c(x,y)=|x_1-y_1|^2$, $G=S^\infty$ be a group of finite permutations acting
by permutations of coordinates. Measures which are invariant with respect to such permutations are called 'exchangeable'.
By the general theory the optimal plan in the set of exchangeable transport plans exists if and only if marginal measures are also exchangeable. d
A detailed description of the solutions to this problem can be found in \cite{Kol-zaev1}.
\end{example}

One can choose a topology for the group $G$ in such way that the functional $(f,x): G \rightarrow \R$, $g \rightarrow (f\circ g)(x)$ is
Borel measurable for any pair $(f,x)\in C_b(X)\times X$.
If the topological group $G$ is \textbf{compact} it is possible to define a transform
$$\bar{f}(x):=\int_G{(f\circ g)(x) d\chi(g)}$$
where $f\in C_b(X)$ and $\chi$ is the left-invariant probability Haar measure on the group $G$.
It is easy to check that $\bar{f}\in C_b(X)$: the integrand is continuous with respect to $x\in X$ and bounded with respect to both $x\in X$
and $g\in G$, which implies the continuity of $\bar{f}(x)$.

Define a new subspace of $C_b$
$$
W_1=\operatorname{span} \{w\in W, f-\int_G f\circ g d\chi(g) \}
$$
We are going to prove that for our problem there is no difference between $W$ and $W_1$.
\begin{proposition}
 $$
 \pi|_{W_1}=0 \iff \pi|_W=0
 $$
\end{proposition}
\begin{proof}
 The implication $\pi|_{W_1}=0 \Rightarrow \pi|_W=0$ is obvious. Prove the opposite one. $\pi|_W=0$ implies that
 $\int_X h\circ g d\pi=\int_X h d\pi$ for any $h \in C_b(X)$, $g \in G$. Then
$$
 \int_X \left( \int_G h\circ g d\chi(g)\right) d\pi = \int_G \left( \int_X h\circ g d\pi\right) d\chi(g)
 =\int_G \left( \int_X h d\pi\right) d\chi(g)=\int_X h d\pi
$$
 Hence
 $$
 \int_X \left( h-\int_G h\circ g d\chi(g)\right) d\pi =0
 $$
 for any $h \in C_b(X)$.
\end{proof}

Let $\overline{W}_1 \in C_b(X)$ be the uniform closure of $W_1$.
Since uniform topology is stronger than $L^1(\pi)$ topology for any $\pi \in \Pi$, it is true that for our problem
there is no difference between $W_1$ and $\overline{W}_1$:
$\pi|_{\overline{W}_1}=0 \iff \pi|_{W_1}=0$ and $\Pi_{W_1}=\Pi_{\overline{W}_1}$.

In the case topological group $G$ is \textbf{compact} it is possible to define a linear operator
$\operatorname{Pr}_{\overline{W}_1}: C_b(X) \rightarrow \overline{W}_1$ as:
$$
\operatorname{Pr}_{\overline{W}_1}(f):=f-\int_G{(f\circ g) d\chi(g)}
$$
where $\chi$ is the left-invariant probability Haar measure on the group $G$
\begin{proposition}
\label{Invariant projection}
If $G$ is a compact group, then the linear operator $\operatorname{Pr}_{\overline{W}_1}$ is a continuous projection on the space $\overline{W}_1$.
\end{proposition}
\begin{proof}

First prove the continuity of the operator with respect to uniform topology on $C_b(X)$.
Let $\sup_{x \in X}|f(x)-h(x)|<\frac{\varepsilon}{2}$ for some fixed $\varepsilon>0$, then
\begin{multline*}
 \sup_{x\in X}|\operatorname{Pr}_{\overline{W}_1}(f-h)(x)|=\sup_{x\in X}|f(x)-h(x)+\int_G{((f-h)\circ g)(x) d\chi(g)}|\leq \\
 \leq \sup_{x\in X}|f(x)-h(x)|+\sup_{x\in X}|\int_G{(f-h)(g(x)) d\chi(g)}| \leq \frac{\varepsilon}{2} + \sup_{x\in X}\int_G{|(f-h)(g(x))| d\chi(g)}\leq\\
 \leq \frac{\varepsilon}{2}+\sup_{x\in X}\int_G{\frac{\varepsilon}{2} d\chi(g)}=\varepsilon
\end{multline*}
Thus the continuity is established.

It is obvious that for any $h \in C_b(X)$ $\operatorname{Pr}_{\overline{W}_1}(h) = h-\int_G{(f\circ g) d\chi(g)} \in \overline{W}_1$.
So to prove that $\operatorname{Pr}_{\overline{W}_1}$ is a projection we only need to show that
$\operatorname{Pr}_{\overline{W}_1}(h)=h$ for any $h\in \overline{W}_1$. If $h=f\circ g_0 -f $ then
$$
\operatorname{Pr}_{\overline{W}_1}(f\circ g_0 - f)=f\circ g_0 - f - \int_G{(f\circ g_0 \circ g) d\chi(g)} +\int_G{(f\circ g) d\chi(g)} = f\circ g_0 - f
$$
If $h= f-\int_G{(f\circ g) d\chi(g)}$
\begin{multline*}
\operatorname{Pr}_{\overline{W}_1}(f-\int_G{(f\circ g) d\chi(g)})=f-\int_G{(f\circ g) d\chi(g)} - \int_G{(f\circ g) d\chi(g)} +\\
+\int_G{\left(\int_G{(f\circ g) d\chi(g)}\right)\circ g) d\chi(g)}= f\circ g - \int_G{(f\circ g) d\chi(g)}
\end{multline*}
Using the linearity we obtain that $\operatorname{Pr}_{\overline{W}_1}=\operatorname{Id}$ on $W_1$.
Due to continuity the operator uniquely extends from a subspace
to its closure, so $\operatorname{Pr}_{\overline{W}_1}=\operatorname{Id}$ on $\overline{W}_1$, which concludes the statement.
\end{proof}

\subsection{Kantorovich duality in the invariant setting}
  The Kantorovich duality relation for the invariant problem can be obtained via substitution of an appropriate subspace $W$ in
  the general equality \ref{bounded Kantorivich duality}.
  But in the case of compact group $G$ this result can be formulated in a more precise way.

  Denote by $\hat{h}:=\operatorname{Pr}_{\overline{W}_1}(h)$ the continuous projection of $h \in C_b(X)$ on $\overline{W}_1$,
  $$\bar{h}:=(Id-\operatorname{Pr}_{\overline{W}_1})(h)=h-\hat{h}=\int_G{(h\circ g) d\chi(g)}$$
  Let $V$ be the image of $C_b(X)$ with respect to this projection. We also have that $C_b(X)=\overline{W}_1\oplus V$ in the sense of direct sum of Banach spaces.
  The same notation will be used for functions $f_k \in C_b(X_k)$ or $f \in F$.
  The role of $\overline{W}_1$ in these cases will be played by the uniform closures of $W_1 \cap C_b(X_k)$ and $W_1 \cap F$ respectively.

  \begin{theorem}
   For the invariant problem with the compact group $G$ and $c \in C_b(X)$ the following formulation of the Kantorovich duality holds if all marginal measures $\mu_k$ are invariant:
    $$
    \inf_{\pi\in \Pi_W}{\int_{X}{c(x) d\pi}}=
    \sup_{f \leq \bar{c}}{\left(\sum_{k=1}^n\int{f_k(x) d\mu_k}\right)}=
    \sup_{\bar{f} \leq \bar{c}}{\left(\sum_{k=1}^n\int{f_k(x) d\mu_k}\right)}
    $$
    where $f \in F=\bigoplus_{k=1}^n C_b(X_k)$.
  \end{theorem}

  \begin{proof}
   From \ref{bounded Kantorivich duality} we have:
   $$
   \inf_{\pi\in \Pi_W}{\int c d\pi}=\sup_{f + \omega \leq c}{\sum_{k=1}^n \int_{X_k}{f_k(x_k) d\mu_k}}.
   $$
   where $f \in F$, $\omega \in W$.
   \begin{multline*}
   \sup_{f + \omega \leq c}{\sum_{k=1}^n \int_{X_k}{f_k(x_k) d\mu_k}}=\\
   =\sup_{f, \omega}\left(\sum_{k=1}^n \int_{X_k}{f_k(x_k) d\mu_k}: \hspace{5pt} \bar{f}+\hat{f}+\omega \leq \bar{c}+\hat{c}\right)=\\
   =\sup_{\bar{f}, \hat{f}, \omega}\left(\sum_{k=1}^n \int_{X_k}{f_k(x_k) d\mu_k}: \hspace{5pt} \bar{f} \leq \bar{c}+(\hat{c}-\omega-\hat{f})\right)
   \end{multline*}
   Note that the maximizing functional doesn't depend on $\overline{W}_1$-part of $f$, thus we can choose $\hat{f}$ arbitrary.
   Hence $\tilde{c}:=\hat{c}-\hat{\omega}-\hat{f}$ is just an arbitrary function from $\overline{W}_1$.
   Inequality $\bar{f}(x) \leq \bar{c}(x)+\tilde{c}(x)$ holds pointwise, so acting on it by an arbitrary element $g \in G$ we obtain:
   $$
   (\bar{f}\circ g)(x) \leq (\bar{c}\circ g+\tilde{c}\circ g)(x) \iff
   \bar{f}(x) \leq \left(\bar{c}+\tilde{c}\circ g\right)(x)
   $$
   for any fixed $x \in X$.
   Thus:
   $$
   \bar{f}(x) \leq (\bar{c}+\tilde{c})(x) \iff \bar{f}(x) \leq \left(\bar{c}+\inf_{g \in G}{(\tilde{c}\circ g)}\right)(x)
   $$
   It can be obtained from the definition of $\overline{W}_1$, that $\inf_{g \in G}{(\tilde{c}\circ g)(x)} \leq 0$ for each fixed point $x$.
   For the elements of $W_1$ it is obvious, and since the uniform convergence implies the pointwise one, it is also true for the elements of the closure.
   Hence the supremum is reached at $\tilde{c} \equiv 0$.
   Finally, we obtain the desired statement:
   \begin{multline*}
   \sup_{\bar{f}, \hat{f}, \omega}\left(\sum_{k=1}^n \int_{X_k}{f_k(x_k) d\mu_k}: \hspace{5pt} \bar{f} \leq \bar{c}+(\hat{c}-\omega-\hat{f})\right)=\\
   =\sup_{\bar{f}, \hat{f}, \omega}\left(\sum_{k=1}^n \int_{X_k}{f_k(x_k) d\mu_k}: \hspace{5pt} \bar{f}\leq \bar{c}+\inf_{g \in G}{((\hat{c}-\omega-\hat{f})\circ g)}\right)=\\
   =\sup_{\bar{f}, \hat{f}}\left(\sum_{k=1}^n \int_{X_k}{f_k(x_k) d\mu_k}: \hspace{5pt} \bar{f} \leq \bar{c}+\inf_{g \in G}(\hat{f}\circ g)\right)=\\
   =\sup_{\bar{f}, \hat{f}}\left(\sum_{k=1}^n \int_{X_k}{f_k(x_k) d\mu_k}: \hspace{5pt} \bar{f} \leq \bar{c}+\hat{f}\right)=\\
   =\sup_{\bar{f}}\left(\sum_{k=1}^n \int_{X_k}{f_k(x_k) d\mu_k}: \hspace{5pt} \bar{f} \leq \bar{c}\right)
   \end{multline*}
  \end{proof}

  Note that in the case of invariant cost function: $c(x)=c(g(x))$ for any $g \in G$, the invariant dual problem
  coincides with the dual problem for the case without additional restrictions. Indeed, if $c$ is invariant then $\bar{c}=c$ and
  we have that
  $$
     \inf_{\pi \in \Pi_W(\mu)}\left\{\int c d\pi\right\}=
    \sup_{f \leq c}{\left\{\sum_{k=1}^n\int{f_k d\mu_k}\right\}}
  $$
  It is known (see, for example, Theorem 2.1.1 from \cite{Rachev1}) that maximizers for the dual problem exist if there is no additional restrictions ($W={0}$).
  Consequently these maximizers are invariant and also appears to be maximizers for the dual invariant problem with the restriction of invariance.

  \begin{corollary}
  If cost function $c$ and all marginal measures $\mu_k$ are invariant with respect to the action of $G$ then
  $$
    \inf_{\pi \in \Pi_W(\mu)}\left\{\int c d\pi\right\}=
    \sup_{f \leq c}{\left(\sum_{k=1}^n\int{f_k d\mu_k}\right)}=
    \inf_{\pi \in \Pi(\mu)}\left\{\int c d\pi\right\}
  $$
  and solution of the invariant dual problem coincides with the solution of the usual Monge-Kantorovich dual problem.
  \end{corollary}
  The same result was obtained in \cite{Moameni} with the use of a different argumentation.

\section*{Appendix}
\textbf{Proof of the Kantorovich duality theorem for the case without additional constraints.}

We are going to prove the next statement
\begin{theorem*}
Let $X_1,...,X_n$, $X=X_1\times \dots \times X_n$ be Polish spaces,
$\mu=\left(\mu_k \in \PP(X_k)\right)$ for $k=1,...,n$, $c \in C_L(\mu)$ Then
$$
\inf_{\pi\in \Pi}{\int c d\pi}=\sup_{f\leq c}{\sum_{k=1}^n \int_{X_k}{f_k(x_k) d\mu_k}}.
$$
where $f \in F=\bigoplus_{i=1}^n C_L(\mu_i)$.
\end{theorem*}

\begin{proof}
 Let $T:F\rightarrow \R$ be a linear functional defined by the formula
 $$
 T(f)=\sum_{i=1}^n \int f_i d\mu_i
 $$
 It is positive and continuous with respect to $\|\cdot\|_L$ - seminorm.
 Let 
 $$
 U(h)=\inf \left\{T(f) : f \geq h\right\}
 $$ be a functional from $C_L(\mu)$ to $\R$.
 It can be proved, that $U$ is subadditive:
 \begin{multline*}
 U(h+g)=\inf_{f \in F} \left\{T(f) : f \geq (h+g)\right\}\leq \\
 \leq \inf \left\{T(f) : f \geq h\right\}+\inf \left\{T(f) : f \geq g\right\}=U(h)+U(g)
 \end{multline*}
 since for any $f_1>h$, $f_2>g$ it is true that $f_1+f_2>h+g$.
 Also $U$ is positively homogeneous: for any $\alpha \in \R^+$
 \begin{multline*}
 U(\alpha h)=\inf_{f \in F} \left\{T(f) : f \geq (\alpha h)\right\}=\\
 = \inf_{f \in F} \left\{T(\alpha f) : f \geq h\right\} = \alpha \inf_{f \in F} \left\{T(f) : f \geq h\right\}=\alpha U(h)
 \end{multline*}
 Thus $U$ is sublinear and is ready to be used in the Hahn-Banach theorem.

 Additionally we will need the following property of $U$: for any $t \in \R$
 $$
 U(t\cdot h)\geq t\cdot U(h)
 $$
 Indeed
 \begin{multline*}
 U(-h)=\inf_{f \in F} \left\{T(f) : f \geq -h\right\}=\inf_{f \in F} \left\{T(f) : -f \leq h\right\}=\\
 =\inf_{f \in F} \left\{T(-f) : f \leq h\right\}=-\sup_{f \in F} \left\{T(f) : f \leq h\right\}\geq\\
 \geq -\inf_{f \in F} \left\{T(f) : f \geq h\right\}= -U(h)
 \end{multline*}
 and combining this result with positive homogeneity we obtain the desired inequality.
 The last inequality $-\sup_{f \in F} \{T(f): f \leq h\} \geq -\inf_{f \in F} \{T(f): f \geq h\}$ follows from the 
 positivity of the functional $T$. Since it is positive,
 it saves order, hence all elements from the image of $\{f: f \leq h\}$ under the map $T$ is not greater then 
 any element from $\{f: f \geq h\}$ under the same map.
 Thus $\sup_{f \in F} \{T(f): f \leq h\} \leq \inf_{f \in F} \{T(f): f \geq h\}$, which is what we want modulo multiplication by $-1$.

 Using the fact that $T\leq U$ on $F$ we can apply Hahn-Banach extension theorem to extend $T$ from $F$
 to the whole space $C_L(\mu)$. Denote such extension as $P$ and prove that the property $P\leq U$ leads to positivity of $P$.

 Assume $P$ is not positive functional. Hence there exists a function $h \in C_L$ such that $h\geq 0$ and $P(h)<0$. The following argument
 $$
 0 < P(-h) \leq U(-h)=\inf \left\{T(f) : f \geq -h\right\}\leq 0
 $$
 leads us to the contradiction.

 Let us define a new linear operator $T_c: \{ f + t c: t\in \R, f\in F\} \rightarrow \R$ such that it coincides with $T$ on $F$: $T_c|_F=T$ and
 coincides with $U$ at the point $-c$: $T_c(-c)=U(-c)$. By linearity of $T_c$ and properties of $U$ it follows 
 that $T_c(t\cdot c) = t\cdot U(c) \leq U(t\cdot c)$.
 Thus $T_c \leq U$ everywhere on its domain and using Hahn-Banach theorem we can extend $T_c$ to the linear functional
 $P_c: C_L(\mu) \rightarrow \R$ such that $P_c|_F=T_c$, $P_c(-c)=U(-c)$, $P_c \leq U$.

 By the construction of linear extensions we have:
 $$
 \sup_{P}{P(-c)}\leq \inf_{f \in F} \left\{T(f) : f \geq -c\right\}
 $$
 where supremum (and infimum) is taken by all possible linear extensions satisfying conditions above (extends $T$ and dominated by $U$).
 Multiplying by $-1$ and using linearity of $T$ and $P$ one can obtain
 $$
 \inf_{P}{P(c)}\geq \sup_{f \in F} \left\{T(f) : f \leq c\right\}
 $$
 Analogously from the equality $P_c(-c)=U(-c)$ using multiplication by $-1$ and linearity of $T$ and $P$ we have
 $$
 P_c(c)= \sup_{f \in F} \left\{T(f) : f \leq c\right\}
 $$
 Since $P_c$ extends $T$ and dominated by $U$
 $$
 \inf_{P}{P(c)}= \sup_{f \in F} \left\{T(f) : f \leq c\right\}
 $$
 This equality differs from the desired duality statement by the fact that infimum is taken over the family of linear operators,
 which are not measures a priori. Therefore in the rest part of the proof we will show that actually these functionals are transport plans.

 Define for any $P$ its restriction $l:=P|_{C_b} \in (C_b(X))^*$ on the dual space for the space of bounded continuous functions on $X$. 
 According to the appropriate version of the Rietz representation theorem $(C_b(X))^*\simeq(C(\beta X))^* \simeq \mathcal{M}(\beta X)$, 
 where $\beta X$ is the Stone-Cech compactification of $X$ and $\mathcal{M}(\beta X)$ is the set of signed measures on it. 
 Both isomorphisms preserve respective norms and positive cones (see e.g. \cite{Bogachev} Th. 7.10.4, 7.10.5). Since $l$ is positive and $\langle l, 1\rangle=1$ 
 the associated measure $\pi$ is a probabilistic.
 
 It is also known that for any $f_i\in C_L(\mu_i)$ $\int f_i d\pi= \langle l, f_i \rangle= \int f_i d\mu_i$. 
 Since $C_L(\mu_i)$ is dense in $L_1(X_i,\mu_i)$ the equality holds true for all integrable (w.r.t. $\mu_i$) functions, 
 in particular for indicator functions of measurable sets on $X_i$.
 
 Let $\pi|_X$ be defined as a measure on $X$ by the formula: $\pi|_X(A)=\pi(A\cap X)$ for all $A\in \beta X$ measurable w.r.t $\pi$. 
 We want to prove that $\pi|_X$ is a probability measure, since it implies that $X$ is a set of full $\pi$-measure. 
 Obviously the total variation of $\pi|_X$ is not greater than one, hence we only need to prove that it is actually not less. 
 Recall that $X$ has a structure of topological direct product: $X=X_1\times\dots\times X_n$, 
 hence there is a well-defined projection $\operatorname{Pr}_i:X\rightarrow X_i$, 
 which pushes forward measure $\pi|_X$ to some measure on $X_i$. 
 Actually the pushforward measure is exactly $\mu_i$: for any $A_i$ measurable w.r.t $\mu_i$
 $$
 ((\operatorname{Pr}_i)_{\#}\pi|_X)(A_i)=\pi|_X(\operatorname{Pr}_i^{-1}(A_i))=
 \pi(\operatorname{Pr}_i^{-1}(A_i))=\int \operatorname{Ind}(A_i) d\mu_i=\mu_i(A_i)
 $$
 In particular $\pi|_X (X)=\pi|_X(\operatorname{Pr}_i^{-1}(X_i))=\mu (X_i)=1$.
 
 Thus we obtained that $l=P|_{C_b}\simeq\pi$ is actually a probability measure on $X$ with marginals $\mu_i$. 
Our next goal is to show that $P$ itself is also a measure.

We are going to use the following seminorm on the space $C_L(\mu)$
 $$
 \| h\|_D=\inf_{f \in F} \left\{T(f) : f \geq |h|\right\}
 $$
 It can be directly checked that this seminorm is actually well-defined (see \cite{Rachev1} for details).
 The associated topology is stronger than $\|\cdot \|_L$ topology:
 $$
 \inf_{f \in F} \left\{T(f) : f \geq |h|\right\}\geq \sup_{\pi \in \Pi}\int |h|d\pi=\| h\|_L
 $$
 but on subspace $C_b \subset C_L$ it is weaker than the topology of uniform convergence:
 $$
 \inf_{f \in F} \left\{T(f) : f \geq |h|\right\}\geq \inf_{f \in F} \left\{T(f) : f \geq \sup_{x \in X}|h(x)|\right\}=
 T(\sup_{x \in X}|h(x)|)=\sup_{x \in X}|h(x)|
 $$
 Note that $P$ is continuous with respect to $\| h\|_D$:
 $$P(|h|) \leq U(|h|)=\| h\|_D$$
 and the restriction of $P$ on $C_b(X)$: $P|_{C_b}=l$ is also continuous.

 The important fact is that $C_b$ is dense in $C_L$ with respect to the seminorm $\| \cdot\|_D$.
 Pick $g\in C_L$, and let $|g|\leq f\in F=\bigoplus_{i=1}^n C_L(\mu_i)$. 
 Let $k \in \N$, $g^k=\min\{k,g\}$ and $g^k_k=\max\{\min\{k,g\},-k\}\in C_b(X)$. Note that $|g^k|\leq f$.
 Show that $\|g-g^k_k\|_D \rightarrow 0$ as $k \rightarrow \infty$:
 $$
 \|g-g^k_k\|_D\leq\|g-g^k\|_D+\|g^k-g^k_k\|_D
 $$
 The fact $\pm g-k \leq f-k = \sum_{i=1}^n\left(f_i-\frac{k}{n}\right)$ together with positivity of the maps $(\cdot)_+:=\max\{\cdot,0\}$
 and $U$ implies
 $$
 \|g-g^k\|_D=U((g-k)_+)\leq U\left(\sum_{i=1}^n\left(f_i-\frac{k}{n}\right)_+\right)\rightarrow 0, \mbox{ as } k \rightarrow \infty
 $$
 $$
 \|g^k-g^k_k\|_D=U((-g^k-k)_+)\leq U\left(\sum_{i=1}^n\left(f_i-\frac{k}{n}\right)_+\right)\rightarrow 0, \mbox{ as } k \rightarrow \infty
 $$
 Convergence here is due to the Lebesgue dominated convergence theorem.
 
Note that $l:=P|_{C_b}$ acts on each function from $C_L(\mu)$ via integration and thus can be seen as a linear operator on $C_L$.
Since integration is continuous in $\|\cdot\|_D$ topology and $\|\cdot\|_D$ is stronger than $\|\cdot\|_L$, we conclude that both extended $l$ and $P$
are continuous linear functionals on $(C_L(\mu),\|\cdot\|_D)$ and coincide on $C_b$.
Using the fact that $C_b$ is dense in $(C_L, \|\cdot\|_D)$ we obtain that extended $l$ and $P$ also coincide on the whole $C_L$
$$
P\simeq\pi \in \Pi(\mu)
$$
Note, that actually $P$ is not just a measure, it is also a transport plan with marginals $\mu_i$. Thus the desired statement is proved.
\end{proof}

\section*{Acknowledgments}
I would like to thank my scientific advisor Alexander Kolesnikov for his valuable ideas, which he shared with me, and
those long discussions that helped me to complete this work.

\end{document}